\documentclass[11pt,a4paper]{amsart}
\usepackage[T1]{fontenc}
\usepackage[latin9]{inputenc}
\usepackage{amssymb}
\usepackage{color}
\usepackage{graphicx,color}
\usepackage{amsmath, amssymb, graphics}

\makeatletter
\numberwithin{equation}{section} 
\numberwithin{figure}{section} 
  \@ifundefined{theoremstyle}{\usepackage{amsthm}}{}
  \theoremstyle{plain}
  \newtheorem{thm}{Theorem}[section]
  \theoremstyle{plain}
  
  \theoremstyle{plain}
  
  \theoremstyle{remark}
  
  \theoremstyle{remark}
  
  \theoremstyle{plain}
  \newtheorem{lem}[thm]{Lemma}
  \newtheorem{defn}[thm]{Definition}

\def\com#1{ \hbox{#1}}

\smallskip
\def\<{{\langle }}
\def\>{{\rangle }}

\def\com#1{ \quad\hbox{#1}\quad}

\smallskip
\def\<{{\langle }}
\def\>{{\rangle }}

\begin{document}

\title[Periods of the Circular RTBP]{On the period of the periodic orbits of the restricted three body problem}
\author{Oscar Perdomo}
\date{\today}

\curraddr{Department of Mathematics\\
Central Connecticut State University\\
New Britain, CT 06050\\}
\email{ perdomoosm@ccsu.edu}


\begin{abstract}
We will show that the period  $T$ of a closed orbit of the planar circular restricted three body problem (viewed on rotating coordinates) depends on the region it encloses. Roughly speaking, we show that, $2 T=k\pi+\int_\Omega g$ where $k$ is an integer, $\Omega$ is the region enclosed by the periodic orbit and $g:\mathbb{R}^2\to \mathbb{R}$ is a function that only depends on the constant $C$ known as the Jacobian integral; it does not depend on $\Omega$. This theorem has a Keplerian flavor in the sense that it relates the period with the space ``swept''  by the orbit. As an application we prove that there is a neighborhood around $L_4$ such that every periodic solution contained in this neighborhood must move clockwise. The same result holds true for $L_5$. 




\end{abstract}

\maketitle

\section{Preliminaries and Introduction}
This paper deals with the motion of three bodies in the plane where the only force being considered is the gravitational force. 
We will assume that the units have been changed so that the gravitational constant is 1, the mass of the first body  is $1-\mu>0$ and the mass of 
the second body is $\mu\le \frac{1}{2}$. Moreover, we will assume that the first body moves according to the formula $p_1(t)=-\mu(\cos t,\sin t)$ and the second body moves according to the formula $p_2(t)=(1-\mu) (\cos t,\sin t)$. The motion of these two bodies is known as the circular solution of the two body problem. Now, let us assume that the third body has a mass $m$ so small that it does not affect the motion of the other two bodies. It follows that if the third body moves according to the formula $(x_1(t),x_2(t))$, then the function $x_1(t)$ and $x_2(t)$ satisfy the following system of differential equations,

{\tiny
\begin{eqnarray*}
\ddot{x}_1&=&\frac{\mu  \left((1-\mu ) \cos (t)-x_1\right)}{\left(\left((\mu -1) \sin (t)+x_2\right){}^2+\left((\mu -1) \cos (t)+x_1\right){}^2\right){}^{3/2}}+\frac{(1-\mu ) \left(-\mu  \cos (t)-x_1\right)}{\left(\left(\mu  \sin (t)+x_2\right){}^2+\left(\mu  \cos (t)+x_1\right){}^2\right){}^{3/2}}\\
\ddot{x}_2&=&\frac{\mu  \left((1-\mu ) \sin (t)-x_2\right)}{\left(\left((\mu -1) \sin (t)+x_2\right){}^2+\left((\mu -1) \cos (t)+x_1\right){}^2\right){}^{3/2}}+\frac{(1-\mu ) \left(-\mu  \sin (t)-x_2\right)}{\left(\left(\mu  \sin (t)+x_2\right){}^2+\left(\mu  \cos (t)+x_1\right){}^2\right){}^{3/2}}
\end{eqnarray*}
}

The problem given by the differential equation above is known as the planar, circular, restricted  three body problem, hereafter, PCR3BP. 

We will be considering periodic solutions of the PCR3BP. It is well-known that the period of a periodic solution of the PCR3BP must be a multiple of $2 \pi$,  \cite{H}. When we change the coordinate system so that it rotates along with the motion of the first two bodies, that is, if we consider the functions $y_1(t)$ and $y_2(t)$ instead of $x_1(t)$ and $x_2(t)$ using the relations
\begin{eqnarray*}
y_1(t)&=& x_1(t) \cos(t)+x_2\sin(t)\\
y_2(t)&=& -x_1(t) \sin(t)+x_2(t)\cos(t)
\end{eqnarray*}

then, it follows that the functions $y_1(t)$ and $y_2(t)$ satisfy the following system of differential equations

\begin{eqnarray}\label{eq2}
\ddot{y}_1 &=&   \frac{\partial \omega}{\partial y_1}+2 \dot{y}_2\\
\ddot{y}_2 &=& \frac{\partial \omega}{\partial y_2}-2 \dot{y}_1 \nonumber
\end{eqnarray}

where,

\begin{eqnarray}\label{rsw}
r_1&=&\sqrt{\left(\mu +y_1\right){}^2+y_2^2}\\
r_2&=&\sqrt{\left(\mu +y_1-1\right){}^2+y_2^2}\\
\omega&=&\frac{1-\mu }{\sqrt{\left(\mu +y_1\right){}^2+y_2^2}}+\frac{\mu }{\sqrt{\left(\mu +y_1-1\right){}^2+y_2^2}}+\frac{1}{2} \left(y_1^2+y_2^2\right)\\
 &=&\frac{1}{2} \left(y_1^2+y_2^2\right)+\frac{1-\mu }{r_1}+\frac{\mu }{r_2}\nonumber
\end{eqnarray}

 
 An important aspect of Equation (\ref{eq2}) is that it is an autonomous system,  and, for this reason, the period of a periodic solution does not have to be a multiple of $2 \pi$.  Regarding the relation between the period of periodic solutions of both systems of differential equations we have that  every periodic solution of the PCR3BP defines a periodic solution of Equation (\ref{eq2}), and a $T$-periodic solutions of Equation (\ref{eq2})  defines a periodic solution of the PCR3BP only when $\frac{\pi}{T}$ is a rational number. A detailed discussion on the periodic solutions of Equation (\ref{eq2}) can be found in \cite{H}.
 
 A direct computation shows that for every solution of Equation (\ref{eq2}) we have that
 
\begin{eqnarray}\label{je}
-\dot{y}_1^2-\dot{y}_2^2 + 2 \omega=C
\end{eqnarray}
 
where $C$ is a constant called the Jacobian Integral \cite{VK}. If follows that every solution of Equation (\ref{eq2}) must be contained in the region 

\begin{eqnarray}\label{jr}
U_C=\{((y_1,y_2)  \,  :\, 2 \omega-C\ge 0 \}\, .
\end{eqnarray}

We point out that since the function $\omega$ is not defined on the set $\{ (-\mu,0),(1-\mu,0)\}$, then, we are not considering these two points to be part of $U_C$ .  

In this paper we prove that if  a regular simple closed curve $(y_1(t),y_2(t))$ is a $T$ periodic solution of Equation (\ref{eq2}), and the region $\Omega$ enclosed by this closed curve is contained in the region $U_C\cup \{(-\mu,0),(1-\mu,0)\}$, then, for some integer $k$ we have that 
  
  $$2 T=k \pi\pm \int_\Omega \Delta \ln (\sqrt{2\omega-C})\, dy_1dy_2$$
  
  The choice of $\pm$ depends on the orientation of the periodic solution.  We extend the result to situations when the curve goes  several times around the point $(-\mu,0)$ or around the point $(1-\mu,0)$. 
  
The proofs of the main results in this paper are very basic since they follow from the Stokes theorem. Despite the simplicity of these results, the author considers  that they are interesting. The key result that allows to use the Stokes theorem is Lemma \ref{le}. The author found the formula in this lemma while he was trying to find an explicit solution of the restricted three body problem. He  was trying to find the Tread-mill Sled of the solution, which is a notion that he has used to explain: the profile curve of surfaces in the Euclidean space with constant mean curvature and helicoidal symmetry \cite{P1}, the profile curve of minimal surface in the Euclidean space and helicoidal symmetry, \cite{P2}, and the shape of ramp with the property that a block will slide down with constant speed \cite{P5}. Also, along with Bennet Palmer, they used the Treamill Sled of a curve to study rotating drops  \cite{P3}, \cite{P4}.   {\color{blue} https://www.youtube.com/watch?v=5puYWqd1xO8} displays a video that explains the notion of Treadmill Sled of a curve.

 \section{The main theorem for simple curves.}   \label{s3}              

Let us start this section with the following lemma. 

\begin{lem}\label{le}
If $\alpha(t)=(y_1(t),y_2(t))$ is a regular curve that is a solution of Equation (\ref{eq2}), with Jacobian integral equal to $C$ (see Equation (\ref{je}) for the definition of $C$), then,

\begin{eqnarray}\label{dys}
\dot{y}_1 &=& \sqrt{2\omega-C}\cos(\theta) \\
\dot{y}_2 &=& \sqrt{2\omega-C}\sin(\theta) 
\end{eqnarray}

where  $\theta(t)$ satisfies the differential equation

\begin{eqnarray}\label{dtheta}
  \dot{\theta} =-2+\frac{\partial f}{\partial y_2} \cos(\theta) -\frac{\partial f}{\partial y_1} \sin(\theta)
\end{eqnarray}  

where $f:U_C\to \mathbb{R}$ is given by $f=\sqrt{2\omega-C}$. See Equation (\ref{jr}) for the definition of $U_C$.
\end{lem}

\begin{proof}
Since the curve $\alpha$ is regular, $\dot{y}_1^2+\dot{y}_2^2$ never vanishes. By the definition of the Jacobian integral we have  $\dot{y}_1^2+\dot{y}_2^2=2\omega-C=f^2$, that is, $|\dot{\alpha}|=f(\alpha)$. Therefore, for some smooth function $\theta$, we have that 

$$\dot{y}_1 = \sqrt{2\omega-C}\cos(\theta)=f\cos(\theta) \com{and}
\dot{y}_2 = \sqrt{2\omega-C}\sin(\theta)=f\sin(\theta) $$

Taking the derivative with respect to $t$ of the two functions above we get

\begin{eqnarray}
\ddot{y}_1&=&(\frac{\partial f}{\partial y_1} f\cos\theta+\frac{\partial f}{\partial y_2} f\sin\theta)\cos\theta-f\dot{\theta}\sin \theta\\
\ddot{y}_2&=&(\frac{\partial f}{\partial y_1} f\cos\theta+\frac{\partial f}{\partial y_2} f\sin\theta)\sin\theta+f\dot{\theta}\cos \theta
\end{eqnarray}

Using Equation (\ref{eq2}) we obtain that

\begin{eqnarray}\label{l1}
(\frac{\partial f}{\partial y_1} f\cos\theta+\frac{\partial f}{\partial y_2} f\sin\theta)\cos\theta-f\dot{\theta}\sin \theta&=&   \frac{\partial \omega}{\partial y_1}+2 f\sin \theta  \\ \label{l2}
(\frac{\partial f}{\partial y_1} f\cos\theta+\frac{\partial f}{\partial y_2} f\sin\theta)\sin\theta+f\dot{\theta}\cos \theta &=& \frac{\partial \omega}{\partial y_2}-2 f \sin \theta
\end{eqnarray}

Multiplying Equation (\ref{l2}) by $\frac{\cos \theta}{f} $ and Equation (\ref{l1}) by $-\frac{\sin \theta}{f} $ and adding them, we obtain the expression for $\dot{\theta}$ given on Equation (\ref{dtheta})
\end{proof}

The following Lemma explain how is the behavior of the function $f$ near the points $(-\mu,0)$ and $(1-\mu,0)$

\begin{lem} \label{l2}
Let $f:U_C\to \mathbb{R}$ be given by $f=\sqrt{2\omega-C}$. If $\beta(t)=(-\mu+\epsilon \cos t, \epsilon \sin t)$ and $n(t)=-(\cos t,\sin t)$, then

$$ \lim_{\epsilon\to 0^+} \frac{\epsilon}{f(\beta(t))} \, \nabla f(\beta(t))\cdot n(t)=\frac{1}{2} $$
Likewise,  if $\beta(t)=(1-\mu+\epsilon \cos t, \epsilon \sin t)$ and $n(t)=-(\cos t,\sin t)$, then

$$ \lim_{\epsilon\to 0^+} \frac{\epsilon}{f(\beta(t))} \, \nabla f(\beta(t))\cdot n(t)=\frac{1}{2} $$

\end{lem}
\begin{proof}
The proof is a direct computation.
\end{proof}

\begin{thm}\label{mtp1}
Let us assume that $\alpha(t)=(y_1(t),y_2(t))$  is a solution of Equation (\ref{eq2}), with Jacobian integral equal to $C$. Let us further assume that 
\begin{itemize}
\item

$\alpha(t)$ is periodic, this is, there exists a positive number $T$ such that  $\alpha(t+T)=\alpha(t)$ for all $t\in\mathbb{R}$.

\item

$\alpha(t)$ is regular, this is,   $|\dot{\alpha}(t)|\ne0$ for all $t\in\mathbb{R}$.
\item

 $\alpha(t)$ is simple, that is, for any $\alpha:[0,T)\to \mathbb{R}^2$ is injective.
 
 \item
 
The region $\Omega$ enclosed by the curve $\alpha$ is contained in the set $U_C$. 
\end{itemize}

We have that

\begin{itemize}
\item
If $\alpha$ goes clockwise around the region $\Omega$ then 
$$2 T= 2 \pi+\iint_\Omega \Delta \ln f dy_1dy_2$$
\item
If $\alpha$ goes counterclockwise around the region $\Omega$ then 
$$2 T= -2 \pi-\iint_\Omega \Delta\ln f dy_1dy_2$$
\end{itemize}
\end{thm}

\begin{proof}
If $\alpha$ goes clockwise around the region $\Omega$, then the normal outer vector of the region $\Omega$ is given by $n=(-\sin \theta,\cos \theta)$ and the total change of the angle $\theta$ is  $-2\pi$. By the Stokes theorem we have that,

$$\iint_\Omega \Delta \ln f= \int_{\partial \Omega}\frac{1}{f}\nabla f \cdot n =\int_0^T(\dot{\theta}+2)\, dt= 2T-2\pi$$

If $\alpha$ goes counterclockwise around the region $\Omega$, then the normal outer vector of the region $\Omega$ is given by $n=(\sin \theta,-\cos \theta)$ and the total change of the angle $\theta$ is  $2\pi$. By Stokes theorem we have that,

$$\iint_\Omega \Delta \ln f= \int_{\partial \Omega}\frac{1}{f}\nabla f \cdot n =\int_0^T(-\dot{\theta}-2)\, dt= -2T-2\pi$$

\end{proof}

\begin{thm}\label{mt2}
Let us assume that $\alpha(t)=(y_1(t),y_2(t))$  is a solution of Equation (\ref{eq2}), with Jacobian integral equal to $C$. Let us further assume that 
\begin{itemize}
\item

$\alpha(t)$ is periodic, this is, there exists a positive number $T$ such that  $\alpha(t+T)=\alpha(t)$ for all $t\in\mathbb{R}$.

\item

$\alpha(t)$ is regular, this is,   $|\dot{\alpha}(t)|\ne0$ for all $t\in\mathbb{R}$.
\item

 $\alpha(t)$ is simple, that is, for any $\alpha:[0,T)\to \mathbb{R}^2$ is injective.
 
 \item
 
The region $\Omega$ enclosed by the curve $\alpha$ contains the point $(-\mu,0)$ and is contained in the set $U_C\cup\{(-\mu,0)\}$. 
\end{itemize}

We have that

\begin{itemize}
\item
If $\alpha$ goes clockwise around the region $\Omega$ then 
$$2 T= \pi+\iint_\Omega \Delta \ln f dy_1dy_2$$
\item
If $\alpha$ goes counterclockwise around the region $\Omega$ then 
$$2 T= - \pi-\iint_\Omega \Delta\ln f dy_1dy_2$$
\end{itemize}
\end{thm}

\begin{proof}
If $\alpha$ goes clockwise around the region $\Omega$, then, the normal outer vector of the region $\Omega$ is given by $n=(-\sin \theta,\cos \theta)$ and the total change of the angle $\theta$ is  $-2\pi$. For any small number $\epsilon$ such that the disk $D_\epsilon$ with center $(-\mu,0)$ and radius $\epsilon$ is contained in $\Omega$ we define $V=\Omega\setminus D_\epsilon$. If we parametrize the boundary of $D_\epsilon$ by $(-\mu+\epsilon \cos(s),\epsilon \sin(s))$, then the outer normal vector of the region $V$ along the boundary of $D_\epsilon$ is $\tilde{n}=(-\cos s, -\sin s)$.  By the Stokes theorem we have that, 

\begin{eqnarray}
\iint_V \Delta \ln f&=& \int_{\partial \Omega}\frac{1}{f}\nabla f \cdot n+\int_{\partial D_\epsilon} \nabla\ln f\cdot \tilde{n} 
\end{eqnarray}

If $\alpha$ goes counterclockwise around the region $\Omega$, then, the normal outer vector of the region $\Omega$ is given by $n=(\sin \theta,-\cos \theta)$ and the total change of the angle $\theta$ is  $2\pi$. By Stokes theorem we have that,

\begin{eqnarray}
\iint_\Omega \Delta \ln f &=& \int_{\partial \Omega}\frac{1}{f}\nabla f \cdot n =\int_0^T(-\dot{\theta}-2)\, dt= -2T-2\pi
\end{eqnarray}

A direct computation, see Lemma \ref{l2}, shows that

$$\lim_{\epsilon\to0}\int_{\partial D_\epsilon} \nabla\ln f\cdot \tilde{n} =\pi$$

Therefore we obtain that 
 \begin{eqnarray*}
\iint_\Omega \Delta \ln f &=& \int_{\partial \Omega}\frac{1}{f}\nabla f \cdot n+\pi\\
&= &\int_0^T(\dot{\theta}+2)\, dt+\pi\\
&=&2T-\pi
\end{eqnarray*}

A similar argument shows the case when the curve $\alpha$ goes counterclockwise around the region $\Omega$.

\end{proof}

Similar arguments show the following theorem. We will omit the proof due to the similarities with the previous one.

\begin{thm}\label{mt3}
Let us assume that $\alpha(t)=(y_1(t),y_2(t))$  is a solution of Equation (\ref{eq2}), with Jacobian integral equal to $C$. Let us further assume that 
\begin{itemize}
\item

$\alpha(t)$ is periodic, this is, there exists a positive number $T$ such that  $\alpha(t+T)=\alpha(t)$ for all $t\in\mathbb{R}$.

\item

$\alpha(t)$ is regular, this is,   $|\dot{\alpha}(t)|\ne0$ for all $t\in\mathbb{R}$.
\item

 $\alpha(t)$ is simple, that is, for any $\alpha:[0,T)\to \mathbb{R}^2$ is injective.
 
\end{itemize}

We have that

\begin{itemize}

 \item
 
If the region $\Omega$ enclosed by the curve $\alpha$ contains the point $(1-\mu,0)$ and is contained in the set $U_C\cup\{(1-\mu,0)\}$,
and $\alpha$ goes clockwise around the region $\Omega$, then 
$$2 T= \pi+\iint_\Omega \Delta \ln f dy_1dy_2$$
\item
 
If the region $\Omega$ enclosed by the curve $\alpha$ contains the point $(1-\mu,0)$ and is contained in the set $U_C\cup\{(1-\mu,0)\}$, and  $\alpha$ goes counterclockwise around the region $\Omega$, then 
$$2 T= - \pi-\iint_\Omega \Delta\ln f dy_1dy_2$$
\item
 
If the region $\Omega$ enclosed by the curve $\alpha$ contains the points $(-\mu,0)$ and $(1-\mu,0)$ and is contained in the set $U_C\cup\{(1-\mu,0),(1-\mu,0)\}$, and  $\alpha$ goes clockwise around the region $\Omega$, then 
$$2 T= \iint_\Omega \Delta\ln f dy_1dy_2$$

\item
 
If the region $\Omega$ enclosed by the curve $\alpha$ contains the points $(-\mu,0)$ and $(1-\mu,0)$ and is contained in the set $U_C\cup\{(-\mu,0),(1-\mu,0)\}$, and  $\alpha$ goes counterclockwise around the region $\Omega$, then 
$$2 T= -\iint_\Omega \Delta\ln f dy_1dy_2$$

\end{itemize}
\end{thm}

 \section{The main theorem for curves with index bigger than one.}   \label{mt}    

In this section we consider some special orbits that go around more than once  before they closed. Let us start this section with a definition that clarifies the type of orbits that we will be considering.

\begin{defn}\label{nsimple}
We will say that a closed curve  $\alpha:[0,T]\to \mathbb{R}^2$, with $\alpha(t)=(z_1(t),z_2(t))$, is $n$-simple around the point $p=(u,v)$ if there exists a simple regular curve  $\beta:[0,T]\to\mathbb{R}^2$, with $\beta(t)=(w_1(t),w_2(t))$ such that 

$$(u,v)+((w_1(t)-u)+i (w_2(t)-v))^n=z_1(t)+i z_2(t) $$

We will call the curve $\beta(t)$ the lifting of $\alpha$.
\end{defn}

\begin{thm}
Let us assume that $\alpha(t)=(y_1(t),y_2(t))$  is a solution of Equation (\ref{eq2}), with Jacobian integral equal to $C$. Let us further assume that $\alpha$ is $T$-periodic and $n$ simple around $(\mu,0)$  with lifting $\beta(t)$  (see Definition \ref{nsimple}).  If we denote by  $f=\sqrt{2 \omega-C}$,  $\tilde{\Omega}$ the region enclosed by $\beta$, and $\tilde{h}=\ln \tilde{f}$ where $\tilde{f}=f\circ \phi$,  and

$$\phi(z_1,z_2)=(\mathbb{R}{\rm e}(\xi),\mathbb{I}{\rm m}(\xi))\com{with} \xi(z_1,z_2)=-\mu+(z_1+\mu+i z_2)^n$$

We have that 

\begin{itemize}

 \item
 
If $\beta(t)$ goes clockwise around $\tilde{\Omega}$ and $\phi(\tilde{\Omega})$ is contained in $U_C\cup\{(-\mu,0)\}$, then 

$$2 T= n \pi+\iint_{\tilde{\Omega}} \Delta \tilde{h} dz_1dz_2$$
\item
 
If $\beta(t)$ goes counterclockwise around $\tilde{\Omega}$ and $\phi(\tilde{\Omega})$ is contained in $U_C\cup\{(-\mu,0)\}$, then 

$$2 T= -n \pi-\iint_{\tilde{\Omega}} \Delta \tilde{h} dz_1dz_2$$

\end{itemize}
\end{thm}
\begin{proof}
Let us denote by $\beta_1(t)$ and $\beta_2(t)$ the two entries of the curve $\beta$, that is, $\beta(t)=(\beta_1(t),\beta_2(t))$. Likewise, we define $\alpha_1(t)$ and $\alpha_2(t)$ and $\phi_1$, $\phi_2$. Notice that if we define $h(y_1,y_2)=\ln f(y_1,y_2)$, then, $\tilde{h}=h\circ \phi$. Using the change rule we obtain that

\begin{eqnarray}\label{ep1}
\frac{\partial \tilde{h}}{\partial z_1}&=&\frac{\partial h }{\partial y_1}\frac{\partial \phi_1}{\partial z_1 }+\frac{\partial h }{\partial y_2}\frac{\partial \phi_2}{\partial z_1 }\\
\frac{\partial \tilde{h}}{\partial z_2}&=&\frac{\partial h }{\partial y_1}\frac{\partial \phi_1}{\partial z_2 }+\frac{\partial h }{\partial y_2}\frac{\partial \phi_2}{\partial z_2 }
\end{eqnarray}

Once again, using the chain rule we obtain that 

\begin{eqnarray}
\dot{\alpha}_1(t)&=&\frac{\partial \phi_1}{\partial z_1 } \dot{\beta}_1(t)+ \frac{\partial \phi_1}{\partial z_2 } \dot{\beta}_2(t)   \\
\dot{\alpha}_2(t)&=&\frac{\partial \phi_2}{\partial z_1 } \dot{\beta}_1(t)+ \frac{\partial \phi_2}{\partial z_2 } \dot{\beta}_2(t) 
\end{eqnarray}

Since the function $\phi$ is an orientation-preserving conformal map, then $\phi_1$ and $\phi_2$  satisfy the Cauchy-Riemann equations, that is,

\begin{eqnarray}\label{ep2}
\frac{\partial \phi_1}{\partial z_1}=\frac{\partial \phi_2}{\partial z_2}\com{and} \frac{\partial \phi_1}{\partial z_2}=-\frac{\partial \phi_2}{\partial z_1}
\end{eqnarray}

Combining the equations (\ref{ep1})-(\ref{ep2}) we get that 

$$ \frac{\partial \tilde{h}}{\partial z_2}\dot{\beta}_1(t)- \frac{\partial \tilde{h}}{\partial z_1}\dot{\beta}_2(t)=
 \frac{\partial h}{\partial y_2}\dot{\alpha}_1(t)- \frac{\partial h}{\partial y_1}\dot{\alpha}_2(t) $$
 
 Let us denote by $D_\epsilon$ a disk with radius epsilon and center at $(-\mu,0)$ contained in $\tilde{\Omega}$ and $\tilde{V}=\tilde{\Omega}\setminus D_\epsilon$. Notice that if  we parametrize the boundary of $D_\epsilon$ by $(-\mu+\epsilon \cos(s),\epsilon \sin(s))$, then the outer normal vector of the region $\tilde{V}$ along the boundary of $D_\epsilon$ is $\tilde{n}=(-\cos s, -\sin s)$ . If $\beta(t)$ goes clockwise around $\tilde{\Omega}$, by Stokes theorem we have that 
 
 \begin{eqnarray}
\iint_{\tilde{V}} \Delta\tilde{h}&=& \int_0^T\frac{\partial \tilde{h}}{\partial z_2}\dot{\beta}_1(t)- \frac{\partial \tilde{h}}{\partial z_1}\dot{\beta}_2(t)\, dt+\int_{\partial D_\epsilon} \nabla \tilde{h} \cdot \tilde{n}
\end{eqnarray}

Taking the limite when $\epsilon$ goes to zero we obtain that

\begin{eqnarray}
\iint_{\tilde{V}} \Delta \tilde{h}&=& \int_0^T\frac{\partial h}{\partial y_2}\dot{\alpha}_1(t)- \frac{\partial h}{\partial y_1}\dot{\alpha}_2(t)\, dt+n\pi\\
&=& \int_0^T(\dot{\theta}+2)+n\pi\\
&=& -2n\pi+2T+n\pi\\
&=& 2T-n\pi
\end{eqnarray}

A similar argument proves the case when $\beta(t)$ goes counterclockwise around $\tilde{\Omega}$
\end{proof}

A similar result holds when the orbit goes around $(1-\mu,0)$. We have

\begin{thm}\label{wi}
Let us assume that $\alpha(t)=(y_1(t),y_2(t))$  is a solution of Equation (\ref{eq2}), with Jacobian integral equal to $C$. Let us further assume that $\alpha$ is $T$-periodic and $n$ simple around $(1-\mu,0)$ with lifting $\beta(t)$  (see Definition \ref{nsimple}).  If  we denote by $f=\sqrt{2 \omega-C}$, $\tilde{\Omega}$ the region enclosed by $\beta$ and $h=\ln \tilde{f}$ where $\tilde{f}=f\circ \phi$,  

$$\phi_n(z_1,z_2)=(\mathbb{R}{\rm e}(\xi_n),\mathbb{I}{\rm m}(\xi_n))\com{with} \xi_n(z_1,z_2)=1-\mu+(z_1+\mu-1+i z_2)^n$$

We have that 

\begin{itemize}

 \item
 
If $\beta(t)$ goes clockwise around $\tilde{\Omega}$ and $\phi(\tilde{\Omega})$ is contained in $U_C\cup\{(1-\mu,0)\}$, then 

$$2 T= n \pi+\iint_{\tilde{\Omega}} \Delta h dz_1dz_2$$
\item
 
If $\beta(t)$ goes counterclockwise around $\tilde{\Omega}$ and $\phi(\tilde{\Omega})$ is contained in $U_C\cup\{(1-\mu,0)\}$, then 

$$2 T= -n \pi-\iint_{\tilde{\Omega}} \Delta h dz_1dz_2$$

\end{itemize}
\end{thm}

\begin{proof}
The proof follows the same arguments as those in the previous theorem.
\end{proof}

Once again, the same argument can be applied when the orbit goes around any given point, in particular one of the Lagrangian points. We have,

\begin{thm}\label{wis}
Let us assume that $\alpha(t)=(y_1(t),y_2(t))$  is a solution of Equation (\ref{eq2}), with Jacobian integral equal to $C$. Let us further assume that $\alpha$ is $T$-periodic and $n$ simple around $(u,v)$ with lifting $\beta(t)$  (see Definition \ref{nsimple}).  If  we denote by $f=\sqrt{2 \omega-C}$, $\tilde{\Omega}$ the region enclosed by $\beta$ and $h=\ln \tilde{f}$ where $\tilde{f}=f\circ \phi$,  

$$\phi_n(z_1,z_2)=(\mathbb{R}{\rm e}(\xi_n),\mathbb{I}{\rm m}(\xi_n))\com{with} \xi_n(z_1,z_2)=u+iv+(z_1-u+i (z_2-v))^n$$

We have that 

\begin{itemize}

 \item
 
If $\beta(t)$ goes clockwise around $\tilde{\Omega}$ and $\phi(\tilde{\Omega})$ is contained in $U_C$, then 

$$2 T= 2 n \pi+\iint_{\tilde{\Omega}} \Delta h dz_1dz_2$$
\item
 
If $\beta(t)$ goes counterclockwise around $\tilde{\Omega}$ and $\phi(\tilde{\Omega})$ is contained in $U_C$, then 

$$2 T= -2 n \pi-\iint_{\tilde{\Omega}} \Delta h dz_1dz_2$$

\end{itemize}
\end{thm}


\section{Examples}

In this section we provide numerically periodic solutions of Equation (\ref{eq2}) and we use them to verify some of the theorems in Section \ref{mt} and Section \ref{s3}.  In order to somehow give these examples a real application, we took $\mu=0.000953875$, which is usually the value considered for the system Sun-Jupiter-Asteroid.  \cite{HJ}

\subsection{Example 1} \label{ex1} Let us consider the initial conditions for Equation (\ref{eq2})

\begin{eqnarray*}
y_1(0)=0.487957127501505 & & y_2(0)=0.84849821703225\\
\dot{y}_1(0)=-0.036041155996589 & & \dot{y}_2(0)=0.02072666577125
\end{eqnarray*}

A direct verification shows that within an error of $10^{-10}$, the solution $\alpha(t)$ is periodic with period $T=6.3036094149426$. In this case the Jacobian integral is $2.9986240063314$. Figure \ref{fig1} shows the orbit of this periodic motion.

\begin{figure}[hbtp]
\begin{center}\includegraphics[width=.4\textwidth]{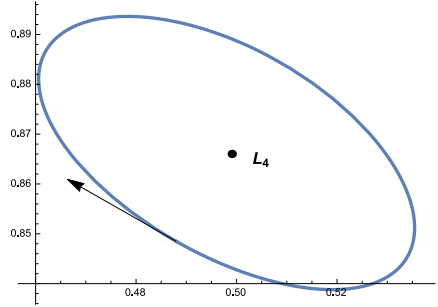}
\end{center}
\caption{Orbit of the periodic solution given in  Example \ref{ex1}}\label{fig1}
\end{figure}

A numerical computation shows that $\iint_\Omega \Delta \ln f=6.32403$ and therefore $2 \pi +\iint_\Omega \Delta \ln f=12.6079$ while $2T=12.607219$. This example verifies Theorem \ref{mtp1} when the orbit goes clockwise.


\subsection{Example 2} \label{ex2} Let us consider the initial conditions  for Equation (\ref{eq2})

\begin{eqnarray*}
y_1(0)=1.01159848498974\quad y_2(0)=0\quad
\dot{y}_1(0)=0\quad  \dot{y}_2(0)=0.26384566980412
\end{eqnarray*}

A direct verification shows that within an error of $10^{-10}$, the solution $\alpha(t)$ is periodic with period $T=0.30139544664015$. In this case the Jacobian integral is $3.0790227765880$. Figure \ref{fig2} shows the orbit of this periodic motion.

\begin{figure}[hbtp]
\begin{center}\includegraphics[width=.4\textwidth]{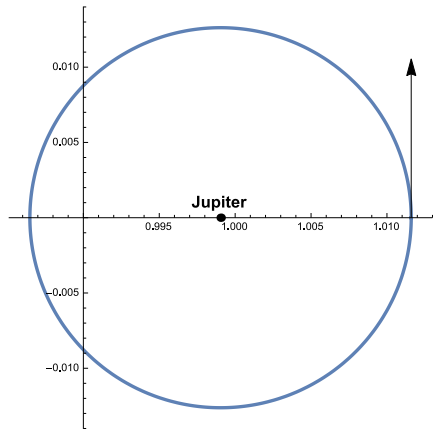}
\end{center}
\caption{Orbit of the periodic solution given in  Example \ref{ex2}}\label{fig2}
\end{figure}

Numerical approximation shows that $\iint_\Omega \Delta \ln f=-3.74433$ and therefore $-\pi -\iint_\Omega \Delta \ln f=0.602739$ while $2T=0.60279089$. This example verifies Theorem \ref{mt3} when the orbit goes counterclockwise and it goes around Jupiter.


\subsection{Example 3} \label{ex3} Let us consider the initial conditions  for Equation (\ref{eq2})

\begin{eqnarray*}
y_1(0)=1.285278846123773 &\quad& y_2(0)=3.401751107285172\\
\dot{y}_1(0)=3.892316782809678&\quad&  \dot{y}_2(0)=-1.47062858674288
\end{eqnarray*}

A direct verification shows that within an error of $10^{-5}$, the solution $\alpha(t)$ is periodic with period $T=5.4912835927302$. In this case the Jacobian integral is $-3.5390576031917$. Figure \ref{fig3} shows the orbit of this periodic motion.

\begin{figure}[hbtp]
\begin{center}\includegraphics[width=.4\textwidth]{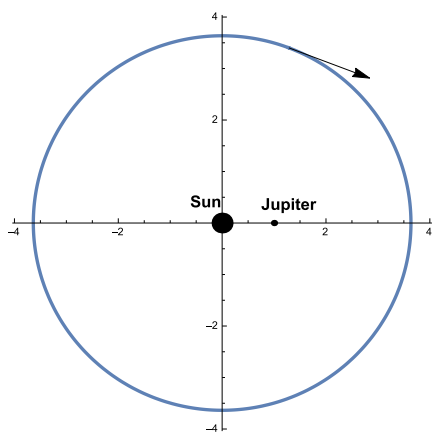}
\end{center}
\caption{Orbit of the periodic solution given in  Example \ref{ex3}}\label{fig3}
\end{figure}

Numerical approximation shows that $\iint_\Omega \Delta \ln f=10.9823$ while $2T=10.982567$. This example verifies Theorem \ref{mt3} when the orbit goes clockwise and it goes around the Sun and Jupiter.


\subsection{Example 4} \label{ex4} Let us consider the initial conditions  for Equation (\ref{eq2})

\begin{eqnarray*}
y_1(0)=0.3964805517652452 &\quad& y_2(0)=-0.07419606744562268\\
\dot{y}_1(0)=0.2120527494053103&\quad&  \dot{y}_2(0)=1.133143493746107
\end{eqnarray*}

A direct verification shows that within an error of $10^{-5}$, the solution $\alpha(t)$ is periodic with period $T=6.2849221865548$. In this case the Jacobian integral is $3.7789562336238$. Figure \ref{fig4} shows the orbit of this periodic motion.

\begin{figure}[hbtp]
\begin{center}\includegraphics[width=.4\textwidth]{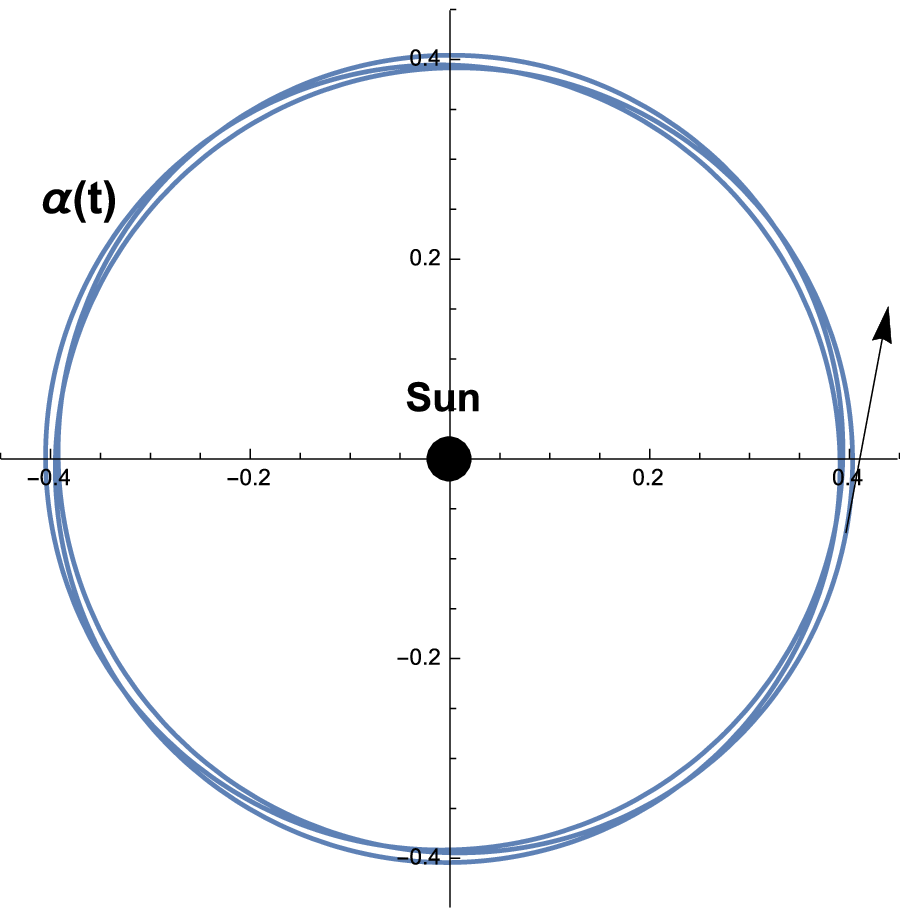}\hskip.2cm\includegraphics[width=.4\textwidth]{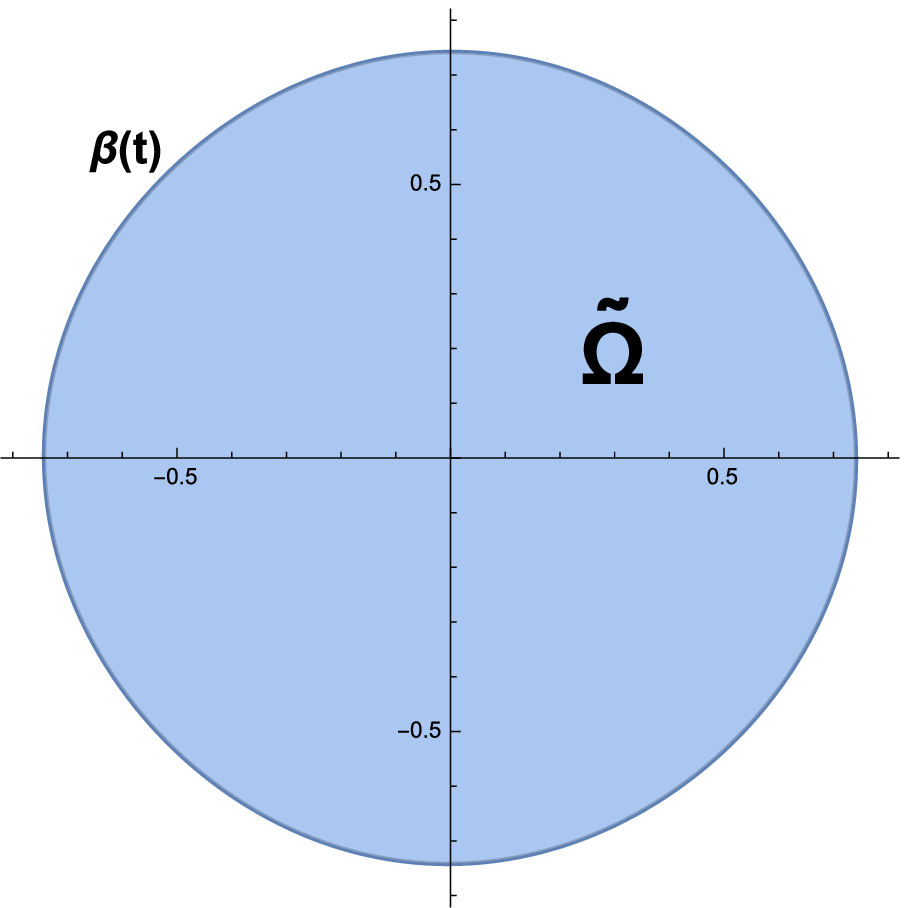}
\end{center}
\caption{On the left we have the orbit of the periodic solution given in  Example \ref{ex4}. This orbit $\alpha$ goes three times around the 
Sun before closing. Also this orbit is $3$-simple around the Sun because the curve $\beta(t)$ shown on the right has the property that $-\mu+(\beta+\mu)^3=\alpha$. In the last expression we had identified $(x,y)$ with the complex number $x+i y$. This figure also shows the region $\tilde{\Omega}$ in $\mathbb{R}^2$ where we integrate on Theorem \ref{wi}. Notice that $\beta$ is the lifting of $\alpha$. }\label{fig4}
\end{figure}

Numerical approximation shows that $\int_{\tilde{\Omega} }\Delta \tilde{h}=-21.9944$ while $2T=12.5698$. Therefore $-3\pi-\int_{\tilde{\Omega}} \Delta \tilde{h}=12.5698$ as predicted by  Theorem \ref{wi} when the orbit goes counterclockwise  around the Sun more than once. Recall that the function $\tilde{h}=\ln (f\circ\phi_3)$. We compute the function $\beta $ by first finding the function $\gamma(t)$ and $q(t)>0$ sucht that $\alpha_1(t)+\mu=q(t) \cos(\gamma(t))$ and $\alpha_2(t)=q(t) \sin(\gamma(t))$. Once we have $q(t)$ and $\gamma(t)$, we obtain that 

$$\beta(t)=\left(\, -\mu+q(t)^\frac{1}{3}\cos\left(\frac{\gamma(t)}{3}\right),q(t)^\frac{1}{3}\sin\left(\frac{\gamma(t)}{3}\right) \, \right)$$

\section{Applications.} The main theorems from this paper provide estimates for the period of the periodic orbits and also allow us to provide necessary conditions for the orbit to be either clockwise or counterclockwise. Example \ref{ex1} showes an orbit near the Lagrangian point $L_4$ orbiting clockwise. The following theorem shows that this happens not only for the Sun-Jupiter system but for any other value of $\mu$.

\begin{thm}
For any $\mu$, there exists an $\epsilon$ such that every periodic solution $\alpha(t)=(y_1(t),y_2(t))$ of the differential equation (\ref{eq2}) with Jacobian integral different from $C_0=3-\mu+\mu^2$, contained in the Disk $D_\epsilon$ with center in $L_4=(\frac{1-\mu}{2},\frac{\sqrt{3}}{2})$ must orbit clockwise. The same result holds true for $L_5=(\frac{1-\mu}{2},-\frac{\sqrt{3}}{2})$
\end{thm}

\begin{proof} A direct computation shows that

\begin{eqnarray}\label{value of h}
\Delta \ln f(\frac{1-\mu}{2},\frac{\sqrt{3}}{2})\, =\, \frac{3}{3+\mu ^2-\mu-C }
\end{eqnarray}

We point out that $C_0$ is a critical value of the function $2\omega$: we can easily check that  $2\omega(L_4)=C_0$, $\nabla \omega(L_4)=0$ and $L_4$ is a local minimum. If $C>C_0$ then by the continuity of $\omega$,  there exists a disc $D_\epsilon$ centered at $L_4$ such that $2\omega<C$. Since $|\dot{\alpha}(t)|^2=2 \omega(\alpha(t))-C$,  no point in $\alpha(t)$ can be in the disk $D_\epsilon$. Now, if $C<C_0$, then by Equation (\ref{value of h}) we have that $\Delta \ln f(L_4)>0$ and using the continuity of $\ln f$, we can find a  disk $D_\epsilon$ centered at $L_4$ such that $\int_{D_\epsilon}\Delta \ln f$ is positive. By Theorems \ref{mtp1} and \ref{wis}, we conclude that the periodic solution $\alpha$ must move clockwise. The same argument proves the result for $L_5$.
\end{proof}




\end{document}